\documentclass[letterpaper,11pt]{amsart}

\usepackage[margin=1.2in]{geometry}
\usepackage{amsmath,amsthm,amssymb}
\usepackage{xspace,xcolor}
\usepackage[breaklinks,colorlinks,citecolor=teal,linkcolor=teal,urlcolor=teal,pagebackref,hyperindex]{hyperref}
\usepackage[alphabetic]{amsrefs}
\usepackage[all]{xy}

\theoremstyle{plain}
\newtheorem{thm}{Theorem}
\newtheorem{lem}[thm]{Lemma}
\newtheorem{prop}[thm]{Proposition}

\theoremstyle{definition}

\newtheorem{eg}[thm]{Example}

\newtheorem{case}{Case}

\newtheorem{problem}[thm]{Problem}

\theoremstyle{remark}
\newtheorem{rmk}[thm]{Remark}

\numberwithin{equation}{section}

\def\Z{{\mathbb Z}}
\def\Q{{\mathbb Q}}

\def\C{{\mathbb C}}

\def\A{{\mathbb A}}
\def\P{{\mathbb P}}

\def\I{\mathcal{I}}
\def\J{\mathcal{J}}
\def\O{\mathcal{O}}

\def\fa{\mathfrak{a}}

\def\fm{\mathfrak{m}}

\def\a{\alpha}
\def\b{\beta}

\def\f{\phi}

\def\k{\kappa}

\def\p{\pi}

\def\s{\sigma}

\def\.{\cdot}
\def\^{\widehat}
\def\~{\widetilde}
\def\o{\circ}
\def\ov{\overline}

\def\rat{\dashrightarrow}

\def\inj{\hookrightarrow}

\def\lru{\lceil}
\def\rru{\rceil}

\def\({\left(}
\def\){\right)}

\newcommand{\ru}[1]{\lru{#1}\rru}

\renewcommand{\and}{ \ \ \text{ and } \ \ }

\DeclareMathOperator{\codim} {codim}

\DeclareMathOperator{\Spec} {Spec}

\DeclareMathOperator{\Aut} {Aut}
\DeclareMathOperator{\Bir} {Bir}

\DeclareMathOperator{\Bl} {Bl}
\DeclareMathOperator{\val} {val}

\DeclareMathOperator{\Ex} {Ex}
\DeclareMathOperator{\mult} {mult}
\DeclareMathOperator{\ord} {ord}

\DeclareMathOperator{\Supp} {Supp}

\DeclareMathOperator{\lct} {lct}

\begin{document}

\title{Fano hypersurfaces and their birational geometry}

\author{Tommaso de Fernex}

\address{Department of Mathematics, University of Utah\\ 
155 South 1400 East, Salt Lake City, UT 48112-0090, USA}
\email{{\tt defernex@math.utah.edu}}


\thanks{The research was partially supported by NSF CAREER grant DMS-0847059
and a Simons Fellowship.}

\thanks{The author would like to thank the referee for useful comments and suggestions.}

\thanks{Compiled on \today. Filename {\tt \jobname}}

\begin{abstract}
We survey some results on the nonrationality and birational rigidity of certain hypersurfaces of Fano type. The focus is on hypersurfaces of Fano index one, but hypersurfaces of higher index are also discussed.
\end{abstract}

\maketitle


\section{Introduction}

This paper gives an account of the main result of 
\cite{dF13}, which states that every smooth complex hypersurface
of degree $N$ in $\P^N$, for $N \ge 4$, is birationally superrigid.
The result is contextualized within the framework of smooth Fano
hypersurfaces in projective spaces and the problem of rationality. 
The paper overviews the history 
of the problem and the main ideas that come into play in its solution, 
from the method of maximal singularities to the use of arc spaces
and multiplier ideals.

Working over fields that are not necessarily algebraically closed, 
we also discuss an extension of a theorem of Segre and Manin
stating that every smooth projective cubic surface of Picard number one
over a perfect field is birationally rigid. 
The proof, which is an adaptation of the arguments of Segre and Manin, 
is a simple manifestation of the method of maximal singularities.

The last section of the paper explores hypersurfaces in $\P^N$ of degree $d < N$. 
We suspect that the result of \cite{dF13} is an extreme case of a more general
phenomenon, and propose two problems which suggest that the 
birational geometry of Fano hypersurfaces should progressively become more rigid
as their degree $d$ approaches $N$.
A theorem of \cite{Kol95} brings some evidence to this phenomenon.

Unless stated otherwise, we work over the field of complex numbers $\C$.
Some familiarity with the basic notions of singularities 
of pairs and multiplier ideals will be assumed;
basic references on the subject are \cite{KM98,Laz04}.

\section{Mori fiber spaces and birational rigidity}

Projective hypersurfaces form a rich class of varieties from the point of view of
rationality problems and related questions. 
We focus on smooth hypersurfaces, and let
\[
X = X_d \subset \P^N
\]
denote a smooth complex projective hypersurface of dimension $N-1$ and degree $d$.
By adjunction, $X$ is Fano (i.e., its anticanonical class $-K_X$ is ample)
if and only if $d \le N$. 

If $d \le 2$ then $X$ is clearly rational with trivial moduli, 
and there is no much else to say. 
However, already in degree $d=3$ the situation becomes rather delicate. 
Cubic surfaces are rational, but cubic threefolds are nonrational by a theorem of 
Clemens and Griffiths \cite{CG72}. 
Moving up in dimension, we find several examples of families
of rational cubics fourfolds \cite{Mor40,Fan44,Tre84,BD85,Zar90,Has99,Has00}, 
with those due to Hassett filling up a countable union 
of irreducible families of codimension 2 in the moduli space of cubic 
hypersurfaces in $\P^5$. By contrast, a conjecture of Kuznetsov \cite{Kuz10}
predicts that the very general cubic fourfold should be nonrational. 
Apart from simple considerations (e.g., rationality of cubic
hypersurfaces of even dimension
containing disjoint linear subspaces of half the dimension)
no much is known in higher dimensions, and
there is no clear speculation on what the picture should be.
In degree $d=4$, we only have Iskovskikh and Manin's theorem on the nonrationality of
$X_4 \subset \P^4$ \cite{IM71}. 

The situation starts to show a more uniform behavior if one bounds the
degree from below in terms of the dimension. 
A result in this direction is due to Koll\'ar \cite{Kol95}.

\begin{thm}
\label{t:kollar}
Let $X = X_d \subset \P^N$ be a very general hypersurface. 
\begin{enumerate}
\item
If $2 \ru{(N+2)/3} \le d \le N$, then $X$ is not ruled (hence is nonrational).
\item
If $3 \ru{(N+2)/4} \le d \le N$, then $X$ is not 
birationally equivalent to any conic bundle.
\end{enumerate}
\end{thm}

This result suggests a certain trend: as the degree approaches (asymptotically) 
the dimension, the birational geometry of the hypersurface tends to `rigidify'. 
This principle can be formulate precisely in the extreme case $d=N$, 
where the geometry becomes as `rigid' as it can be.

A \emph{Mori fiber space} is a normal $\Q$-factorial 
projective variety with terminal singularities,
equipped with a morphism of relative Picard number one
with connected fibers of positive dimension such that the anticanonical class is
relatively ample.
Examples of Mori fiber spaces are conic bundles and Del Pezzo fibrations. 
A Fano manifold with Picard number one can be regarded as a Mori 
fiber space over $\Spec \C$. 

\begin{thm}
\label{t:X_N-P^N}
Let $X = X_N \subset \P^N$ be any (smooth) hypersurface.
If $N \ge 4$, then every birational map from $X$ to a Mori fiber space $X'/S'$
is an isomorphism (and in fact a projective equivalence). 
In particular, $X$ is nonrational.
\end{thm}

We say that $X_N \subset \P^N$, for $N \ge 4$, is \emph{birationally superrigid}. 
In general, a Fano manifold $X$ of Picard number one 
is said to be \emph{birationally rigid}
if every birational map $\f$ from $X$ to a Mori fiber space $X'/S'$ 
is, up to an isomorphism, 
a birational automorphism of $X$; it is said to be \emph{birationally superrigid}
if any such $\f$ is an isomorphism.\footnote{This definition can 
be generalized to all Mori fiber spaces, see \cite{Cor95}.}

Theorem~\ref{t:X_N-P^N} has a long history, tracing back to the work of Fano
on quartic threefolds \cite{Fan07,Fan15}. Let $X = X_4 \subset \P^4$. Fano claimed that
$\Bir(X) = \Aut(X)$, a fact that alone suffices to show that $X$ is nonrational as
$\Aut(X)$ if finite and $\Bir(\P^3)$ is not. Fano's method is inspired to Noether's 
factorization of planar Cremona maps. The idea is to look at the
indeterminacy locus of a given birational self-map $\f \colon X \rat X$.
If $\f$ is not an isomorphism, then the base scheme $B \subset X$ 
of a linear system defining $\f$ must be `too singular' with respect to the
equations cutting out $B$ in $X$. As this is impossible, one concludes that
$\f$ is a regular automorphism. 

Fano's proof is incomplete. The difficulty that Fano had to face
is that, differently from the surface case where the multiplicities
of the base scheme are a strong enough invariant to quantify how `badly singular'
the map is, in higher dimension one needs to dig further into 
a resolution of singularities to extract the relevant information. 

The argument was eventually corrected and completed in \cite{IM71}. In their paper, 
Iskovskikh and Manin only look at the birational group $\Bir(X)$, 
but it soon became clear that the proof itself leads to the stronger conclusion that $X$ 
is birationally superrigid. In fact, the very definition of birational superrigidity
was originally motivated by their work.\footnote{Mori fiber spaces 
are the output of the minimal model program
for projective manifolds of negative Kodaira dimension. It is natural
to motivate the notion of birational rigidity also from this point of view:
a Mori fiber space is birationally
rigid (resp., superrigid) if, within its own birational class, 
it is the unique answer of the program
up to birational (resp., biregular) automorphisms
preserving the fibration.}

Following \cite{IM71}, significant work 
has been done throughout the years to extend this result to higher dimensions, 
starting from Pukhlikov who proved it first 
for $X_5 \subset \P^5$ \cite{Puk87},
and then in all dimensions under a suitable condition of `local regularity' on the equation
defining the hypersurface \cite{Puk98}. Some low dimensional cases were
established in \cite{Che00,dFEM03}, and the complete proof of Theorem~\ref{t:X_N-P^N} 
was finally given in \cite{dF13}. 

While in this paper we focus on smooth projective hypersurfaces, 
the birational rigidity problem has been extensively studied for many other Fano
varieties and Mori fiber spaces, especially in dimension 3. 
There is a large literature on the subject that is too vast to be included here. 
For further reading, a good place to start is \cite{CR00}.

The study of birational rigidity has also ties with other birational properties
of algebraic varieties such as unirationality and rational connectedness.
The work of Iskovskikh and Manin was originally motivated 
by the L\"uroth problem, which asked whether unirational varieties
are necessarily rational. It was known by work of Segre \cite{Seg60} that 
there are smooth quartic threefolds $X_4 \subset \P^4$ that are unirational, 
and it is easy to see that all smooth cubic threefolds $X_3 \subset \P^4$ are unirational. 
The results of \cite{IM71,CG72} gave the first counter-examples to the L\"uroth problem.

Birational rigidity also relates to stability properties. 
A recent theorem of Odaka and Okada \cite{OO} proves that 
any birationally superrigid Fano manifold of index 1 
is slope stable in the sense of Ross and Thomas \cite{RT06}.


\section{Cubic surfaces of Picard number one}

Before Fano's idea could be made work in dimension three, 
Segre found a clever way to apply Noether's method once more to dimension two. 
Cubic surfaces are certainly rational over the complex numbers, but
they may fail to be rational when the ground field is not algebraically closed. 
The method of Noether works perfectly well, in fact, 
to prove that every smooth projective cubic surface 
of Picard number one over a field $\k$ is nonrational \cite{Seg51}. 
Later, Manin observed that if the field is perfect then the proof can be adapted 
to show that if two such cubic surfaces are birational equivalent, then 
they a projectively equivalent \cite{Man66}.\footnote{The 
hypothesis in Manin's theorem that $\k$ be perfect 
can be removed, cf.\ \cite{KSC04}.}
For a thorough discussion of these results, see also \cite{KSC04}.

The theorems of Segre and Manin extend rather straightforwardly
to the following result, which implies that cubic surfaces of Picard number one
are \emph{birationally rigid} (over their ground field). 

\begin{thm}
\label{t:X_3-P^3}
Let $X_\k \subset \P^3_\k$ be a smooth cubic surface 
of Picard number one over a perfect field $\k$.
Suppose that there is a birational map $\f_\k \colon X_\k \rat X_\k'$ where
$X_\k'$ is either a Del Pezzo surface of Picard number one, or a conic bundle over a curve $S_\k'$.
Then $X'_\k$ is a smooth cubic surface 
of Picard number one, and there is a birational automorphism $\b_\k \in \Bir(X_\k)$
such that $\f_\k \o \b_\k \colon X_\k \to X'_\k$ is a projective equivalence. 
In particular, $X_\k$ is nonrational.
\end{thm}

\begin{proof}
Fix an integer $r' \ge 1$ and a divisor $A'_\k$ on $X'_\k$, given
by the pullback of a very ample divisor on $S'_\k$,
such that $-r'K_{X'_\k} + A'_\k$ is very ample.
Here we set $S'_\k = \Spec\k$ and $A'_\k = 0$ if $X_\k'$ is a Del Pezzo surface of Picard number one.

Since $X_\k$ has Picard number one, its Picard group is generated
by the hyperplane class, which is linearly equivalent to
$-K_{X_\k}$. Then there is a positive integer $r$ such that
\[
(\f_\k)_*^{-1}(-r'K_{X'_\k} + A'_\k) \sim -rK_{X_\k}.
\]

Let $\ov\k$ be the algebraic closure of $\k$,
and denote $X = X_{\ov\k}$, $X' = X'_{\ov\k}$, $S' = S'_{\ov\k}$, $A' = A'_{\ov\k}$ and
$\f = \f_{\ov\k}$. Note that $A'$ is zero if $\dim S' = 0$, and
is the pullback of a very ample divisor on $S'$ if $\dim S' = 1$.
Let $D' \in |-r'K_{X'} + A'|$ be a general element, and let 
\[
D = \f^{-1}_*D' \in |-rK_X|.
\]

We split the proof in two cases.

\begin{case}
Assume that $\mult_x(D) > r$ for some $x \in X$.
\end{case}

The idea is to use these points of high multiplicity to construct a suitable
birational involution of $X$ (defined over $\k$) that, 
pre-composed to $\f$, untwists the map.
This part of the proof is the same as in
the proof of Manin's theorem, and we only sketch it.
The construction is also explained 
in \cite{KSC04}, to which we refer for more details. 

The Galois group of $\ov\k$ over $\k$ acts on the base points of $\f$ 
and preserves the multiplicities of $D$ at these points. Since
$D$ belongs to a linear system with zero-dimensional base locus and
$\deg D = 3r$ (as a cycle in $\P^3)$, there are at most two points 
at which $D$ has multiplicity larger than $r$, and the union of these points
is preserved by the Galois action.
If there is only one point $x \in X$ (not counting infinitely near ones), 
then $x$ is defined over $\k$. 
Otherwise, we have two distinct points $x,y$ on $X$ whose union 
$\{x,y\} \subset X$ is defined over $\k$. 

In the first case, consider the rational map $X \rat \P^2$ given by 
the linear system $|\O_X(1) \otimes \fm_x|$ (i.e., 
induced by the linear projection $\P^3 \rat \P^2$ with center $x$). 
The blow-up $g \colon \~X \to X$ of $X$ at $x$ resolves the indeterminacy of the map, 
and we get a double cover $h \colon \~X \to \P^2$. The Galois group of this cover
is generated by an involution $\~\a_1$ of $\~X$, which descends to a birational
involution $\a_1$ of $X$. 
In the second case, consider the map $X \rat \P^3$ given by 
the linear system $|\O_X(2) \otimes \fm_x^2 \otimes \fm_y^2|$. 
In this case, we obtain a double cover $h \colon \~X \to Q \subset \P^3$ where 
now $g \colon \~X \to X$ is the blow-up of $X$ at $\{x,y\}$ and $Q$ is
a smooth quadric surface. As before, we denote by $\~\a_1$ the Galois involution 
of the cover and by $\a_1$ the birational involution induced on $X$. 
In both cases, $\a_1$ is defined over $\k$. Therefore the composition 
\[
\f_1 = \f \o \a_1 \colon X \rat X'
\]
is defined over $\k$ and hence is given by a linear system in $|-r_1K_X|$ for some $r_1$.
A point $x$ with $\mult_x(D) > r$ cannot be an Eckardt point, and thus
is a center of indeterminacy for $\f_1$. 

In either case, we have $r_1 < r$. 
To see this, let $E$ be the exceptional divisor of $g \colon \~X \to X$,
and let $L$ be the pullback to $\~X$ of the hyperplane class
of $\P^2$ (resp., of $Q \subset \P^3$) by $h$. Note that $L \sim g^*(-K_X) - E$
by construction, and $g_*\~\a_1{}_*E \sim -sK_X$ for some $s \ge 1$
since it is supported on a nonempty curve (by Zariski's Main Theorem)
that is defined over $\k$.  
If $m$ is the multiplicity of $D$ at $x$ (and hence at $y$ in the second case)
and $\~D$ is the proper transform of $D$ on $\~X$, then 
$\~D  + (m-r)E \sim rL$.
Applying $(\~\a_1)_*$ to this divisor and pushing down to $X$, we obtain 
${\a_1}_*D \sim -r_1K_X$ where $r_1= r - (m-r)s < r$ since $m > r$.
Therefore, this operation lowers the degree of the equations defining the map.

Let $D_1 = {\f_1}_*^{-1}D' \in |-r_1K_X|$.
If $\mult_x(D_1) > r_1$ for some $x \in X$, then we proceed as before
to construct a new involution $\a_2$, and proceed from there.
Since the degree decreases each time, 
this process stops after finitely many steps.
It stops precisely when, letting
\[
\f_i = \f \o \a_1 \o \dots \o \a_i \colon X \rat X'
\]
and $D_i = {\f_i}_*^{-1}D' \in |-r_iK_X|$, we have 
$\mult_x(D_i) \le r_i$ for every $x \in X$.
Note that $\f_i$ is defined over $\k$. 
Then, replacing $\f$ by $\f_i$, we reduce to the next case.

\begin{case}
Assume that $\mult_x(D) \le r$ for every $x \in X$.
\end{case}

Taking a sequence of blow-ups, we obtain a resolution of indeterminacy
\[
\xymatrix{
& Y \ar[dl]_p \ar[dr]^q & \\
X \ar@{-->}[rr]^\f && X' 
}
\]
with $Y$ smooth.
Write
\begin{align*}
K_Y + \tfrac 1{r'} D_Y 
&= p^*(K_X + \tfrac 1{r'}D) + E' \\ 
&= q^*(K_{X'} + \tfrac 1{r'}D') + F'
\end{align*}
where $E'$ is $p$-exceptional, $F'$ is $q$-exceptional,
and $D_Y = p_*^{-1}D = q_*^{-1}D'$.
Since $X'$ is smooth and $D'$ is a general hyperplane section, 
we have $F' \ge 0$ and $\Supp(F') = \Ex(q)$.
Note that $K_{X'} + \tfrac 1{r'}D'$ is nef. 
Intersecting with the image in $Y$ of a general complete intersection
curve $C \subset X$ we see that $(K_X + \tfrac 1{r'}D) \. C \ge 0$, 
and this implies that $r \ge r'$. 

Next, we write
\begin{align*}
K_Y + \tfrac 1{r} D_Y 
&= p^*(K_X + \tfrac 1{r}D) + E \\ 
&= q^*(K_{X'} + \tfrac 1{r}D') + F
\end{align*}
where, again, $E$ is $p$-exceptional and $F$ is $q$-exceptional. 
The fact that $\mult_x(D) \le r$ for all $x \in X$ implies that $E \ge 0$. 
Intersecting this time with the image in $Y$ of a general complete intersection
curve $C'$ in a general fiber of $X' \to S'$, 
we get $(K_{X'} + \tfrac 1{r}D') \. C' \ge 0$, 
and therefore $r = r'$. Note also that $E = E'$ and $F = F'$. 

The difference $E - F$ is numerically equivalent to the pullback of $A'$.
In particular, $E-F$ is nef over $X$ and is numerically trivial over $X'$. 
Since $p_*(E-F) \le 0$, the Negativity Lemma, applied to $p$, implies that $E \le F$. 
Similarly, 
since $q_*(E-F) \ge 0$, the Negativity Lemma, applied to $q$, implies that $E \ge F$. 
Therefore $E = F$. This means that $A'$ is numerically trivial,
and hence $S' = \Spec\ov\k$. 
Furthermore, we have $\Ex(q) \subset \Ex(p)$, and therefore
Zariski's Main Theorem implies that the inverse map
\[
\s = \f^{-1} \colon X' \rat X
\]
is a morphism. 

To conclude, just observe that if $S'_\k = \Spec \k$ then $X'_\k$ must have Picard number one. 
But $\s$, being the inverse of $\f$, is defined over $\k$. 
It follows that $\s$ is an isomorphism, as otherwise it would increase the 
Picard number. Therefore $X'_\k$ is a smooth cubic surface
of Picard number one. Since we can assume without loss of generality to have picked $r' = 1$
to start with, we conclude that, after the reduction step performed in
Case~1, $\f$ is a projective equivalence defined over $\k$.
The second assertion of the theorem follows by taking
$\b_\k$ given by $\a_1 \o \dots \o \a_i$ over $\k$.
\end{proof}

\section{The method of maximal singularities}

The proof of Theorem~\ref{t:X_3-P^3} already shows the main features
of the \emph{method of maximal singularities}. 

The reduction performed in Case~1 of the proof is 
clearly inspired by Noether's untwisting process used to
factorize planar Cremona maps into quadratic transformations
\cite{Noe72,Cas01}. 
This procedure has been generalized in higher dimensions to build 
the \emph{Sarkisov's program}, which provides
a way of factorize birational maps between Mori fiber spaces
into \emph{elementary links}, see \cite{Cor95,HM13}.

The discussion of Case~2 of the proof generalizes
to the following property, due to \cite{IM71,Cor95}.\footnote{For a comparison, 
one should notice how similar the arguments are.
We decided to use the same exact wording when the argument is the same 
so that the differences will stand out.}

\begin{prop}[Noether--Fano Inequality]
\label{p:NF}
Let $\f \colon X \rat X'$ be a birational map 
from a Fano manifold $X$ of Picard number one to a Mori fiber space $X'/S'$. 
Fix a sufficiently divisible integer $r'$ and a sufficiently ample divisor
on $S'$ such that if $A'$ is the pullback of this divisor to $X'$ then
$-r'K_{X'} + A'$ is a very ample divisor
(if $S' = \Spec\C$ then take $A' = 0$).
Let $r$ be the positive rational number
such that $\f_*^{-1}(-r'K_{X'}+A) \sim_\Q -rK_X$, and
let $B \subset X$ be the base scheme of the linear system 
$\f_*^{-1}|-r'K_{X'}+A| \subset |-rK_X|$. 
If the pair $(X,\tfrac 1r B)$ is canonical, then $r = r'$ and $\f$ is an isomorphism.
\end{prop}

\begin{proof}
Let
\[
\xymatrix{
& Y \ar[dl]_p \ar[dr]^q & \\
X \ar@{-->}[rr]^\f && X' 
}
\]
be a resolution of singularities. Note that 
the exceptional loci $\Ex(p)$ and $\Ex(q)$ have pure codimension 1.
Fix a general element $D' \in |-r'K_{X'}+A|$ and let $D_Y = q_*^{-1}D$ (which is the same
as $q^*D$) and $D = p_*D_Y$. Note that $D_Y = p_*^{-1}D$ and $D =\f_*^{-1}D' \in |-rK_X|$.

Write
\begin{align*}
K_Y + \tfrac 1{r'} D_Y 
&= p^*(K_X + \tfrac 1{r'}D) + E' \\ 
&= q^*(K_{X'} + \tfrac 1{r'}D') + F'
\end{align*}
where $E'$ is $p$-exceptional and $F'$ is $q$-exceptional. 
Since $X'$ has terminal singularities and $D'$ is a general hyperplane section, 
we have $F' \ge 0$ and $\Supp(F') = \Ex(q)$.
Note that $K_{X'} + \tfrac 1{r'}D'$ is numerically 
equivalent to the pullback of $A'$, which is nef. 
Intersecting with the image in $Y$ of a general complete intersection
curve $C \subset X$ we see that $(K_X + \tfrac 1{r'}D) \. C \ge 0$, 
and this implies that $r \ge r'$. 

Next, we write
\begin{align*}
K_Y + \tfrac 1{r} D_Y 
&= p^*(K_X + \tfrac 1{r}D) + E \\ 
&= q^*(K_{X'} + \tfrac 1{r}D') + F
\end{align*}
where, again, $E$ is $p$-exceptional and $F$ is $q$-exceptional. 
Assume that the pair $(X,\tfrac 1r B)$ is canonical. 
Since $D$ is defined by a general element of the linear system 
of divisors cutting out $B$, 
and $r \ge 1$, it follows that
$(X,\tfrac 1r D)$ is canonical. This means that $E \ge 0$. 
Intersecting this time with the image in $Y$ of a general complete intersection
curve $C'$ in a general fiber of $X' \to S'$, 
we get $(K_{X'} + \tfrac 1{r}D') \. C' \ge 0$, 
and therefore $r = r'$. Note also that $E = E'$ and $F = F'$. 

The difference $E - F$ is numerically equivalent to the pullback of $A'$.
In particular, $E-F$ is nef over $X$ and is numerically trivial over $X'$. 
Since $p_*(E-F) \le 0$, the Negativity Lemma, applied to $p$, implies that $E \le F$. 
Similarly, 
since $q_*(E-F) \ge 0$, the Negativity Lemma, applied to $q$, implies that $E \ge F$. 
Therefore $E = F$. This means that $A'$ is numerically trivial, and hence
$S' = \Spec\C$ and $X'$ is a Fano variety of Picard number one. 
Furthermore, we have $\Ex(q) \subset \Ex(p)$, 

By computing the Picard number of $Y$ in two ways (from $X$ and from $X'$),
we conclude that $\Ex(p)= \Ex(q)$, and thus the difference 
$p^*D - q^*D'$ is $q$-exceptional. Since $D$ is ample, 
this implies that $\f$ is a morphism. 
Since $X$ and $X'$ have the same Picard number and $X'$ is normal, 
it follows that $\f$ is an isomorphism. 
\end{proof}

\begin{rmk}
A more general version of this property gives a criterion
for a birational map $\f \colon X \rat X'$ between two Mori fiber spaces
$X/S$ and $X'/S'$ to be an isomorphism
preserving the fibration. Given the correct statement, 
the proof easily adapts to this setting. For more details, see 
\cite{Cor95,dF02} (the proof in \cite{Cor95} uses, towards the end, 
some results from the minimal model program;
this is replaced in \cite{dF02} by an easy computation of
Picard numbers similar to the one done at the end of the proof of the proposition).
\end{rmk}

The idea at this point is to relate this condition on the singularities of
the pair $(X,\tfrac 1r B)$ to intersection theoretic invariants 
such as multiplicities, which can be easily related to 
the degrees of the equations involved when, say, $X$ is
a hypersurface in a projective space. 

If $X$ is a smooth surface and $D$ is an effective divisor,
then $(X,\tfrac 1r D)$ is canonical if and only if
$\mult_x(D) \le r$ for every $x \in X$.
In higher dimension, however, being canonical
cannot be characterized by a simple condition on multiplicities.

The way \cite{IM71} deals with this problem is by carefully keeping track 
of all valuations and discrepancies along the exceptional
divisors appearing on a resolution of singularities.
The combinatorics of the whole resolution, encoded in a suitable graph
which remembers all centers of blow-up, becomes an essential 
ingredient of the computation. 
This approach has been used to study birational rigidity problems for several years
until Corti proposed in \cite{Cor00} an alternative approach
based on the Shokurov--Koll\'ar Connectedness Theorem. 
Corti's approach has led to a significant simplification of the proof
of Iskovskikh--Manin's theorem, and has provided a starting point 
for setting up the proof of Theorem~\ref{t:X_N-P^N}.

\section{Cutting down the base locus}

Let $X = X_4 \subset \P^4$, and suppose that $\f \colon X \rat X'$ is a birational
map to a Mori fiber space $X'/S'$ which is not
an isomorphism. Using the same notation as in Proposition~\ref{p:NF}, 
it follows that the pair $(X,\tfrac 1r B)$ is not canonical. 
The following property, due to \cite{Puk98}, 
implies that the pair is canonical away from a finite set.

\begin{lem}
\label{l:mult1}
Let $X \subset \P^N$ be a smooth hypersurface, and let $D \in |\O_X(r)|$. 
Then $\mult_C(D) \le r$ for every irreducible curve $C \subset X$.
\end{lem}

\begin{proof}
Let $g \colon X \to \P^{N-1}$ be the morphism induced by
projecting from a general point of $\P^N$, and write $g^{-1}(g(C)) = C \cup C'$.
The residual component $C'$ has degree $(d-1)\deg(C)$ where $d = \deg(X)$. 
Taking a sufficiently general projection, the ramification divisor 
intersects $C$ transversely at $(d-1)\deg(C)$ distinct points $x_i$, 
which are exactly the points of intersection $C \cap C'$. 
If $\mult_C(D) > r$, then we get
\[
\deg(D|_{C'}) \ge \sum_i \mult_{x_i}(D|_{C'}) > r (d-1)\deg(C) = \deg(D|_{C'}),
\]
a contradiction.
\end{proof}

Therefore there is a 
prime exceptional divisor $E$ on some resolution $f \colon \~X \to X$,
lying over a point $x \in X$, such that 
\[
\tfrac 1r\. \ord_E(B) > \ord_E(K_{\~X/X}).
\]
where $K_{\~X/X}$ is the relative canonical divisor.
In the left hand side we regard $\ord_E$ as a valuation on the function field of $X$, 
and $\ord_E(B) = \ord_E(\I_B)$ denotes the smallest valuation of an element
of the stalk of the ideal sheaf $\I_B \subset \O_X$ of $B$ at the center of valuation $x$. 

Corti's idea, at this point, is to take a general hyperplane section
$Y \subset X$ through $x$. This has two effects:
\begin{enumerate}
\item
the restriction $B|_Y$ of the base scheme $B$ is a zero-dimensional scheme, and
\item
the pair $(Y,\tfrac 1r B|_Y)$ is not log canonical. 
\end{enumerate}
The first assertion is clear. 
Let us discuss why~(b) is true. 
Suppose for a moment that the proper transform $\~Y \subset \~X$ of $Y$ intersects 
(transversely) $E$, and let $F$ be an irreducible component of $E|_{\~Y}$. 
By adjunction, we have 
\[
K_{\~Y/Y} = (K_{\~X/X} + \~Y - f^*Y)|_{\~Y}. 
\]
Since $\ord_E(Y) \ge 1$ and $\ord_F(B|_Y) \ge \ord_E(B)$, we have
\[
\tfrac 1r\. \ord_F(B|_Y) > \ord_F(K_{\~Y/Y}) + 1,
\]
and this implies~(b). 
In general, we cannot expect that $\~Y$ intersects $E$. 
Nevertheless, the Connectedness Theorem tells us that,
after possibly passing to a higher resolution, $\~Y$ will intersect some 
other prime divisor $E'$ over $X$, with center $x$, such that 
\[
\tfrac 1r\.\ord_{E'}(B) + \ord_{E'}(Y) > \ord_{E'}(K_{\~X/X}) + 1.
\]
Then the same computation using the adjunction formula produces a divisor $F$
over $Y$ satisfying the previous inequality.

The property that $(Y,\tfrac 1r B|_Y)$ is not log canonical can be
equivalently formulated in terms of log canonical thresholds.
It says that the log canonical threshold $c = \lct(Y,B|_Y)$ of the pair $(Y,B|_Y)$
satisfies the inequality
\[
c < 1/r.
\]
The advantage now is that we know how to compare log canonical thresholds
to multiplicities. The following result is due to \cite{Cor00,dFEM04}.

\begin{thm}
\label{t:lct-e}
Let $V$ be a smooth variety of dimension $n$, let
$Z \subset V$ be a scheme supported at a closed point $x \in V$, and let $c = \lct(V,Z)$. 
Then $Z$ has multiplicity
\[
e_x(Z) \ge (n/c)^n.
\]
\end{thm}

For the purpose of establishing birational rigidity, one only needs the case $n=2$ of this theorem, 
which is the case first proved by Corti. The case $n=2$ can be deduced from a more general 
formula which can be easily proven by induction on the number of blow-ups
needed to produce a log resolution. Here we sketch the proof in all dimension 
which, although perhaps less direct, has the advantage of explaining the nature of the result
as a manifestation of the classical inequality
between arithmetic mean and geometric mean. 

\begin{proof}[Sketch of the proof of Theorem~\ref{t:lct-e}]
The proof uses a flat degeneration to monomial
ideals. It is easy to prove the theorem in this case. 
If $\fa \subset \C[u_1,\dots,u_n]$ is a $(u_1,\dots,u_n)$-primary monomial ideal
then the log canonical threshold $c$ can be computed directly 
from the Newton polyhedron. This allows to reduce to the case in which
$\fa = (u_1^{a_1},\dots,u_n^{a_n})$, where
the log canonical threshold is equal to $\sum 1/a_i$ and the
Samuel multiplicity is equal to $\prod a_i$.
In this special case, the stated inequality is just the usual inequality
between arithmetic mean and geometric mean. 
\end{proof}

Applying the case $n=2$ of this theorem to our setting, we get
\[
e_x(B|_Y) \ge (2/c)^2 > 4r^2,
\]
which is impossible because $B|_Y$, being cut out on $Y$ by equations of degree $r$, 
is contained in a zero-dimensional complete intersection scheme of degree $4r^2$. 
This finishes the proof of Iskovskikh--Manin's theorem.

\section{Beyond connectedness: first considerations}

Consider now the general case $X = X_N \subset \P^N$. 
We would like to apply Corti's strategy for all $N \ge 4$. 
Again, we use the notation of Proposition~\ref{p:NF}
and assume the existence of a non-regular birational map $\f \colon X \rat X'$. 
Then there is a prime divisor $E$ on a resolution $f\colon \~X \to X$ such that
\[
\tfrac 1r\. \ord_E(B) > \ord_E(K_{\~X/X}).
\]
and $E$ maps to a closed point $x \in X$ by Lemma~\ref{l:mult1}.

In order to cut down the base scheme $B$ to a zero-dimensional scheme, 
we need to restrict to a surface. Let $Y \subset X$ be the
surface cut out by $N-3$ general hyperplane sections through $x$. 
Then $B|_Y$ is zero dimensional. 
As before, the Connectedness Theorem implies that $(Y,\tfrac 1r B|_Y)$
is not log canonical at $x$, and we get the inequality $e_x(B|_Y) > 4r^2$. 

If $X$ is sufficiently \emph{general} in moduli, then it 
contains certain cycles of low degree and high multiplicity at $x$, 
and the inequality is still sufficient to conclude that $X$ is birationally superrigid, 
as shown in \cite{Puk98}.
However, if $X$ is \emph{arbitrary} in moduli, then for $N \ge 5$ the inequality 
is not strong enough to give a contradiction 
as now $B$ is cut out by equations of degree $r$ on a surface of degree $N$ (rather than 4). 

The issue is that we are cutting down several times, but we are not
keeping track of this. Morally, we should expect that as we
keep cutting down, the singularities of the pair get `worse' at each step. 
One can try to measure this by looking at the multiplier ideal
of the pair. 
If we cut down to a general hyperplane section
$H \subset X$ through $x$, then the pair $(H,\frac 1r B|_H)$ is not 
log canonical, and this implies that its multiplier ideal 
$\J(H,\frac 1r B|_H)$ is nontrivial at $x$. 
In fact, we can do better: if we set $c = \lct(H,B|_H)$
then $\J(H,cB|_H)$ is nontrivial at $x$,
which is a stronger condition since $c < 1/r$. 
The question is: What happens when we cut further down?
Optimally, the multiplier ideal will `get deeper' at each step and
we can use this information to get a better bound on the multiplicity of $B|_Y$. 

Suppose for instance that the proper transform $\~Y$ intersects (transversely) $E$, 
and let $F$ be a component of $E|_{\~Y}$. Since $Y$ has codimension $N-3$, 
the adjunction formula gives, this time,
\[
c\. \ord_F(B|_Y) - (N-4)\.\ord_F(\fm_{Y,x}) \ge \ord_F(K_{\~Y/Y}) + 1.
\]
This condition can be interpreted in the language of multiplier ideals
by saying that, locally at $x$, 
\[
(\fm_{Y,x})^{N-4} \not\subset \J(Y,cB|_Y).
\]
Applying Theorem~\ref{t:J-e} below, we get
\[
e_x(B|_Y) \ge 4(N-3)/c^2 > 4(N-3)r^2.
\]
For $N \ge 4$, this contradicts the fact that $B|_Y$ is contained
in a zero-dimensional complete intersection scheme of degree $Nr^2$. 

The result we have applied is the following reformulation of Theorem~2.1 of \cite{dFEM03}, 
which in turn is a small variant of Theorem~\ref{t:lct-e}.

\begin{thm}
\label{t:J-e}
Let $V$ be a smooth variety of dimension $n$, let
$Z \subset V$ be a scheme supported at a closed point $x \in V$, and let $c > 0$. 
Assume that $(\fm_{V,x})^k \not\subset \J(V,cZ)$ locally near $x$.
Then $Z$ has multiplicity
\[
e_x(Z) \ge (k+1)(n/c)^n.
\]
\end{thm}

Unfortunately, this computation breaks down if $\~Y$ is disjoint from $E$, 
and there is no stronger version of the Connectedness Theorem to fix it. 
In fact, in general the multiplier ideal simply fails to `get deeper'. 
This is already the case in the following simple example.

\begin{eg}
Let $D = ( y^2=x^3) \subset \A^2$ and $c = 5/6$. Then $\J(\A^2,cD) = \fm_{\A^2,0}$. 
\begin{enumerate}
\item
If $L \subset \A^2$ is a general line through the origin, then $\J(L,cD|_L) = \fm_{L,0}$. 
\item
If $L = (y=0) \subset \A^2$, then $\J(L,cD|_L) = (\fm_{L,0})^2$. 
\end{enumerate}
\end{eg}

\section{The role of the space of arcs}

Let us discuss the example a little further. 
The multiplier ideal of $(\A^2,cD)$ can be computed by taking the well-known
log resolution $f \colon \~X \to \A^2$ of the cusp given by a 
sequence of three blow-ups. The exceptional divisor $E$ extracted by the third 
blow-up computes the log canonical threshold of the pair (which is $c = 5/6$), 
and is responsible for the nontrivial multiplier ideal. 
That is, we have $c\.\ord_{E}(D) = \ord_{E}(K_{\~X/\A^2}) + 1$, and 
\[
\J(\A^2,cD) = f_*\O_{\~X}(\ru{K_{\~X/\A^2} - cf^*D}) = f_*\O_{\~X}(-E) = \fm_{\A^2,0}.
\]
No matter how we choose $L$, 
the proper transform $\~L$ will always be disjoint from $E$, 
so we cannot rely on the computation done in the previous section.  
What makes the choice of $L$ in case~(b) more special 
is that in this case the proper transform of $L$ on the first blow-up
$\Bl_0\A^2$ contains the center of $\ord_{E}$ on the blow-up. 
The intuition is that this choice brings $\~L$ `closer' to $E$, at 
least `to the first order'. 

In order to understand what is really happening, 
we work with formal arcs. 
Given a variety $X$, the \emph{arc space} is given, set theoretically, by
\[
X_\infty = \{\a \colon \Spec \C[[t]] \to X \}.
\]
This space inherits a scheme structure from his description 
as the inverse limit of the jet schemes
which parametrize maps $\Spec \C[t]/(t^{m+1}) \to X$, 
and comes with a morphism $\p \colon X_\infty \to X$ 
mapping an arc $\a(t)$ to $\a(0) \in X$. 
Note that $X_\infty$ is not Noetherian, is not of finite type, 
and does not have finite topological dimension.

Given a resolution $f \colon \~X \to X$ and a smooth prime divisor $E$ on it, 
we consider the diagram
\[
\xymatrix{
\~\p^{-1}(E) \ar[d] \ar@{}[r]|-\subset & \~X_\infty \ar[d]_{\~\p} \ar[r]^{f_\infty}
& X_\infty \ar[d]^\p & C_X(E) \ar[d] \ar@{}[l]|-\supset \\
E \ar@{}[r]|-\subset & \~X \ar[r]^f & X  & x \ar@{}[l]|-\ni 
}
\]
where $f_\infty$ is the map on arc spaces given by composition and
$C_X(E)$ is the \emph{maximal divisorial set} of $E$, defined by
\[
C_X(E) = \ov{f_\infty(\~\p^{-1}(E))}.
\]
This set is irreducible and only depends on the valuation $\ord_E$.
The following theorem due to \cite{ELM04} is the key to relate
this construction to discrepancies and multiplier ideals. 

\begin{thm}
\label{t:ELM}
Suppose that $X$ is a smooth variety. 
\begin{enumerate}
\item
The generic point $\a$ of $C_X(E)$ defines a valuation 
$\ord_\a \colon \C(X)^* \to \Z$, and this valuation coincides with $\ord_E$. 
\item
The set $C_X(E)$ has finite topological codimension in $X_\infty$, and 
this codimension is equal to $\ord_E(K_{\~X/X}) + 1$. 
\end{enumerate}
\end{thm}

It is easy to guess how the valuation is defined:
the generic point of $C_X(E)$ is a $K$-valued arc $\a \colon \Spec K[[t]] \to X$,
the pullback map $\a^* \colon \O_{X,\p(\a)}\to K[[t]]$ extends
to an inclusion of fields $\a^* \colon \C(X) \inj K((t))$, and the valuation
is obtained by simply composing with the valuation $\ord_t \colon K((t))^* \to \Z$. 
The computation of the codimension of $C_X(E)$ uses the
description of the fibers of the maps at the jet levels $f_m\colon \~X_m \to X_m$
and is essentially equivalent to the change-of-variable formula in motivic integration.

For our purposes, the advantage of working with divisorial sets in arc spaces is that,
given a subvariety $Y \subset X$ containing the center of $\ord_E$, even 
if the proper transform $\~Y$ is disjoint from $E$, the arc space $Y_\infty$
(which is naturally embedded in $X_\infty$) always intersects $C_X(E)$. 
We can then pick an irreducible component 
\[
C \subset (Y_\infty \cap C_X(E)).
\]
In general, $C$ itself may not be a maximal divisorial set. However,  
it is not too far from it. In the language of \cite{ELM04}, one says that
$C$ is a \emph{cylinder} in $Y_\infty$, which essentially means that $C$
is cut out by finitely many equations (maximal divisorial 
sets in arc spaces of smooth varieties are examples of cylinders). The upshot is
that the generic point $\b$ of $C$ defines a valuation $\val_\b$ of $\C(Y)$, 
and we can find a prime divisor $F$ over $Y$ and a positive integer $q$ such that
\[
\val_\b = q\.\val_F \and \codim(C,Y_\infty) \ge q\.(\ord_F(K_{\~Y/Y}) + 1).
\]
Now we can start relating multiplier ideals, since 
we can easily compare $q\.\ord_F$ to $\val_E$, 
and we control the equations cutting out $Y_\infty$ inside $X_\infty$
and hence how the codimension of $C$ compares to that of $C_X(E)$.

In order to control how the multiplier ideal behaves under restriction, 
and to show that it gets deeper, we need to ensure that
if $E$ has center $x \in X$ then
$q\.\ord_F(\fm_{Y,x}) = \val_E(\fm_{X,x})$. 
In general, there is only an inequality. 
A tangency condition on $Y$ is the first step to achieve this. 
This condition alone is not enough in general, but 
it suffices in the homogeneous setting, when $X = \A^n$ and the valuation
$\ord_E$ is invariant under the homogeneous $\C^*$-action on 
a system of coordinates centered at $x$. 
The following result is proved in \cite{dF13}.

\begin{thm}
\label{t:restr}
Let $X = \A^n$, let $Z \subset \A^n$ be a closed subscheme, and let $c > 0$.
Assume that there is a prime divisor $E$ on some resolution $\~X \to \A^n$
with center a point $x \in X$ such that
\begin{enumerate}
\item
$c\.\ord_E(Z) \ge \ord_E(K_{\~X/\A^n}) + 1$, and 
\item
the valuation $\ord_E$ is invariant under the homogeneous $\C^*$-action on 
a system of affine coordinates centered at $x$. 
\end{enumerate}
Let $Y = \A^{n-k} \subset\A^n$ be a linear subspace 
of codimension $k$ through $x$ that is tangent to
the direction determined by a general point of the center of $E$ in $\Bl_x\A^n$. 
Then 
\[
(\fm_{Y,x})^k \not\subset \J(Y,cZ|_Y).
\]
\end{thm}

Note that the conclusion of the theorem is precisely the 
condition assumed in Theorem~\ref{t:J-e}. 
Unfortunately we cannot apply this theorem directly to where we left in the proof of
Theorem~\ref{t:X_N-P^N} because we first need to reduce to a homogeneous setting. 
This reduction is probably the most 
delicate part of the proof of Theorem~\ref{t:X_N-P^N}.
To make the reduction, we use linear projections 
to linear spaces and flat degenerations to homogeneous ideals. The idea of using linear
projections first appeared in \cite{Puk02}, where a proof of birational rigidity
was proposed but turned out to contain a gap.
Theorem~\ref{t:J-e} is hidden behind the proof of a certain inequality on log canonical 
thresholds under generic projection, just like it was in the proof of
the main theorem of \cite{dFEM03}.
Nadel's vanishing theorem is used in the end to draw the desired contradiction. 
This part of the proof is technical and goes beyond the purpose of this note; 
for more details, we refer the interested reader to \cite{dF13}.

\section{What to expect for Fano hypersurfaces of higher index}

Birational rigidity fails for hypersurfaces $X = X_d \subset \P^N$ of degree $d < N$
for a very simple reason:
each linear projection $\P^N \rat \P^k$, for $0 \le k \le N-d$, 
induces a Mori fiber structure on the hypersurface. 
If the center of projection is not contained in $X$, then the general fiber
of the Mori fiber space is given by a Fano hypersurface of index $N+1-d-k$. 
Note that, by the Lefschetz Hyperplane Theorem, if $k \ge N/2$ (and $d \ge 2$) then 
the center of projection cannot be contained in $X$. 

In low dimensions, other Mori fiber spaces may appear. 
Consider for instance, the case $X = X_3 \subset \P^4$. 
As explained above, $X$ admits 
fibrations in cubic surfaces over $\P^1$, each induced by a linear projection
$\P^4 \rat \P^1$. 
Additionally, for every line $L \subset X$ the linear
projection $\P^4 \rat \P^2$ centered at $L$ induces, birationally, a
conic bundle structure of $X$ onto $\P^2$. 
Furthermore, the projection from any point $x \in X$ induces 
a birational involution of $X$ which swaps the two sheets of the
rational cover $X \rat \P^3$. 
However, these two last constructions are more specific of the low
dimension and low degree of the hypersurface, 
and do not generalize when the degree and dimension get larger.  

With this in mind, it is natural to consider the following problem.

\begin{problem}
\label{prob:1}
Find a (meaningful) function $g(N)$ such that
for every $X_d \subset \P^N$, with $g(N) \le d \le N$,
the only Mori fiber spaces birational to $X$ are those induced by the
linear projections $\P^N \rat \P^k$ for $0 \le k \le N-d$.
\end{problem}

\begin{rmk}
Taking $g(N) = N$ will of course work for $N \ge 4$
by Theorem~\ref{t:X_N-P^N}. The problem is to determine, if it exists, 
a better function which includes Fano hypersurfaces of higher index. 
Already proving that $g(N) = N-1$ works for $N \gg 1$ would be very interesting.\footnote{A 
partial solution
to Problem~\ref{prob:1} in the special case $g(N) = N-1$ has been recently announced in \cite{Puk}.}
\end{rmk}

A similar problem is the following.

\begin{problem}
\label{p:2}
Find a (meaningful) function $h(m,N)$ such that there is no $X_d \subset \P^N$, with
$h(m,N) \le d \le N$, birational to a Mori fiber space of fiber dimension $\le m$
(other than $X \to \Spec\C$ if $m = N-1$).
\end{problem}

As a first step, one can try to find solutions to these problems
that work for general (or, even, very general) 
hypersurfaces. Part~(b) of Theorem~\ref{t:kollar}
gives a great solution to the case $m=1$ 
of Problem~\ref{p:2} for very general hypersurfaces. 

\end{document}